\documentclass[11pt,a4paper,reqno]{amsart}

\usepackage[UKenglish]{babel}

\usepackage{amsmath,mathbbol,amssymb,mathbbol}
\usepackage[T1]{fontenc} \usepackage{mathpazo}
\usepackage{comment,cite}

\usepackage{amsbsy}
\usepackage{amsfonts}
\usepackage{graphicx,subfigure,alphalph}
\usepackage[shortlabels]{enumitem}
\usepackage{verbatim}
\usepackage{hyperref}
\hypersetup{
    colorlinks=true,
    linkcolor=blue,
    citecolor=red,
    filecolor=magenta,      
    urlcolor=cyan,
    pdftitle={Domination of semigroups generated by regular forms},
    linktocpage=true,
}
\usepackage{ifthen}
\usepackage{tikz}
\usepackage{tikz-cd}
\usepackage{psfrag}
\usepackage{epsfig}
\usepackage{color}

 \graphicspath{{figures/}}
 \DeclareGraphicsExtensions{.jpg,.eps,.ps,.pdf,.png}

\makeatletter

\newcommand{\restrict}[1]{|_{#1}}

\def\E{\mathcal E}

\def\R{{\mathbb R}}

 \def\N{{\mathbb N}}

 \def\C{{\mathbb C}}

 \def\cA{{\mathfrak A}}
 \def\E{{\mathfrak a}}

 \def\Re{{\rm Re }}
 
 \def\Cap{{\rm cap}}
 \def\alg{{\rm alg}}

\newcommand{\norm}[1]{\lVert#1\rVert}

\usepackage{color,xcolor}
\definecolor{blue}{cmyk}{1,0.70,0.10,0.50} 

\newtheorem{theorem}{Theorem}[section]
\newtheorem{lemma}[theorem]{Lemma}
\newtheorem{proposition}[theorem]{Proposition}

\newtheorem{remark}[theorem]{Remark}

\numberwithin{equation}{section}

\DeclareMathOperator{\sgn}{sgn}

\begin{document}

\title{Domination of semigroups generated by regular forms}

\author{Sahiba Arora}

\address{S.~Arora, Institut f\"ur Analysis, Fakult\"at Mathematik, TU Dresden, 01062 Dresden, Germany}
\email{sahiba.arora@mailbox.tu-dresden.de}

\author{Ralph Chill}

\address{R.~Chill, Institut f\"ur Analysis, Fakult\"at Mathematik, TU Dresden, 01062 Dresden, Germany}
\email{ralph.chill@tu-dresden.de}

\author{Jean-Daniel Djida}

\address{J.-D.~Djida, African Institute for Mathematical Sciences (AIMS), P.O. Box 608, Limbe Crystal Gardens, South West Region, Cameroon}
\email{jeandaniel.djida@aims-cameroon.org}

\date{\today} 
	
	

\begin{abstract}
    We give a representation for regular forms associated with domi\-nated $C_0$-semigroups which, in turn, characterises the domination of $C_0$-semi\-groups associated with regular forms. In addition, we prove a relationship between the positivity of (dominated) $C_0$-semi\-groups and the locality of the associated forms.
\end{abstract}

\maketitle

\section{Introduction}

In the theory of semigroups on the Hilbert space $L^2$ generated by bilinear forms, there is a well-known characterisation due to Ouhabaz \cite[Theorem~2.21]{Ou04}, which characterizes domination of semigroups in terms of the gene\-rating forms. This result is a consequence of another characterisation, namely of the situation when a semigroup (orbit) on an abstract Hilbert space leaves a closed and convex subset invariant \cite[Theorem~2.2]{Ou04}. In fact, the Beurling-Deny criteria -- which characterise positivity and $L^\infty$-contractivity of a semigroup -- are also consequences of this abstract result. 

In this short note, we revisit the domination of semigroups generated by forms. The well-known characterization of domination \cite[Theorem~2.21]{Ou04} states that if $T$ and $\widehat{T}$ are two semigroups on $L^2_{\mu}(\Omega)$ generated by quadratic forms $\E$ and $\widehat{\E}$ respectively such that $\widehat{T}$ is positive, then  the semigroup $T$ is dominated by $\widehat{T}$  if and only if $D(\E )$ is an ideal of $D(\widehat{\E})$ and
\begin{equation*} 
    \widehat{\E} (|u|,|v|) \leq \Re\ \E (u,v)
\end{equation*}
for every $u$, $v\in D(\E )$ with $u\bar{v}\geq 0$. In our first main result (Theorem~\ref{main}), we give an integral representation theorem for the difference $\E-\widehat{\E}$ under the assumption that the measure space $\Omega$ is a topological measure space and that the form $\E$ is in some sense regular;  thereby showing explicitly which type of form perturbations lead to domination. In particular, we see that some non-local perturbations may also lead to domination (compare with the proof of \cite[Theorem~4.1]{ArWa03b}). Secondly, we show that (Theorem~\ref{thm.local}) if $\widehat{\E}$ is local and if $T$ is dominated by $\widehat{T}$, then the positivity of $T$ is equivalent to the locality of $\E$. In particular, we give a simpler proof of \cite[Theorem 4.3]{Ak18}.

\subsection*{A characterisation of domination}
Throughout, we let $(\Omega ,\cA ,\mu )$ be a topological measure space. By this we mean that $\Omega$ is a topological space, $\cA$ is the Borel $\sigma$-algebra, and $\mu$ is a (positive) Borel measure on $(\Omega ,\cA )$. Without loss of generality, we assume that $\mu$ has full support in the sense that there exists no nonempty open subset $U\subseteq\Omega$ which has zero measure. All vector spaces in this note are {\em complex} vector spaces. 

We now state our main result:
\begin{theorem} \label{main}
 Let $\E :D(\E )\times D(\E )\to\C$ and $\widehat{\E} : D(\widehat{\E} ) \times D(\widehat{\E} ) \to\C$ be two sesquilinear, Hermitian, closed, accretive, and densely defined forms on $L^2_\mu (\Omega )$ such that the associated self-adjoint $C_0$-semigroups (denoted by  $T$ and $\widehat{T}$ respectively) are real. Assume that the semigroup $\widehat{T}$ is positive, and that the form $\E$ is $\mathcal{A}$-regular for some closed, Hermitian, and unital subalgebra $\mathcal{A}$ of $C^b (\Omega )$. Further, let $\widetilde{\Omega}$ denote the compactification of $\Omega$ associated with $\mathcal A$.
 
 Then $T$ is dominated by $\widehat{T}$, in the sense that 
 \[
  |T(t)u| \leq \widehat{T} (t)|u|,\qquad \text{ for every } u\in L^2_\mu (\Omega ),
 \]
 if and only if $D(\E )$ is an ideal of $D(\widehat{\E} )$ and there exists a (positive) Hermitian Borel measure $\nu$ on $\widetilde{\Omega}\times \widetilde{\Omega}$ such that $\nu$ is absolutely continuous with respect to the $\E$-capacity, the equality
 \begin{equation} \label{eq.domination}
  \E (u,v) = \widehat{\E} (u,v) + \int_{\widetilde{\Omega}\times \widetilde{\Omega}} u(x) \overline{v(y)} \; d\nu (x,y)
 \end{equation}
 holds for every $u,v\in D(\E )$, and 
 \begin{equation} \label{eq.domination.2}
 \Re\int_{\widetilde{\Omega}\times \widetilde{\Omega}} u(x) \overline{v(y)} \; d\nu (x,y) \geq \widehat{\E} (|u|,|v|) - \Re\ \widehat{\E} (u,v)
 \end{equation}
 for every $u,v\in D(\E )$ which satisfy $u\bar{v}\geq 0$ on $\Omega$.
\end{theorem}

The rest of this section will be dedicated to explaining the terminology used in the statement and that we need to prove Theorem~\ref{main}. Then, in Section~\ref{section:proof} we provide the proof of Theorem~\ref{main}. In Section~\ref{section:locality}, we turn our attention to local forms. Finally, we leave the reader with some additional remarks in Section~\ref{section:remarks}.

\subsection*{Notation and Preliminaries}
If $u$ is a vector in a Banach lattice $E$, then we use the usual notation $u^+$ and $u^-$ to denote the positive and negative parts of $u$ respectively. Recall that $u=u^+-u^-$ and the modulus of $u$ satisfies $|u|=u^++u^-$.

 Let $\E : D(\E ) \times D(\E ) \to\C$ be a sesquilinear form, where $D(\E )$ is a dense  subspace of $L^2_\mu (\Omega )$ known as the {\em form domain}. We say that the form $\E$ is {\em Hermitian} if for every $u$, $v\in D(\E )$, 
 \[
 \E (u,v) = \overline{\E (v,u)} .
\]
The form $\E$ is said to be {\em accretive} if for every $u\in D(\E )$,
 \[
  \Re\,\E (u):= \Re\,\E (u,u) \geq 0 ,
 \]
and an accretive form is called {\em closed} if the form domain $D(\E )$ is complete with respect to the norm
 \[
      \|u\|_{D(\E)}:= \sqrt{\Re\ \E(u) + \|u\|_{L^2_\mu}^2}\qquad (u\in D(\E)).
 \]
Clearly, $D(\E )$ embeds continuously into $L^2_\mu(\Omega)$ when equipped with this norm. 
Given two Hermitian forms $\E$ and $\widehat{\E}$, we say that $D(\E )$ is an {\em ideal} of $D(\widehat{\E})$ if 
\begin{enumerate}[\upshape (a)]
    \item the implication $u\in D(\E)\Rightarrow |u| \in D(\widehat{\E})$ holds and
    \item whenever $u\in D(\E )$ and $v\in D(\widehat{\E})$, the inequality $|v|\leq |u|$ implies $v\sgn u\in D(\E )$.
\end{enumerate}
In particular, the above ideal property is stronger than the notion of lattice ideal in Banach lattices. For a situation when the two definitions are equivalent, see \cite[Proposition~2.23]{Ou04}.

For every Hermitian, accretive, and closed sesquilinear form, the operator given by
\begin{align*}
 D(A) & := \{ u\in D (\E ) |\ \exists\ f\in L^2_\mu (\Omega )\ \forall\ v\in D(\E ) : \E (u,v) = \langle f,v\rangle_{L^2_\mu} \} , \\
 Au & := f ,
\end{align*}
is self-adjoint and positive semi-definite. This operator is the negative generator of a self-adjoint contraction semigroup $T := (T(t))_{t\geq 0}$. Actually, there is a one-to-one correspondence between Hermitian, closed, accretive, and densely defined forms, the self-adjoint and positive semi-definite operators, and the self-adjoint contraction semigroups. For all these facts, we refer the reader to standard monographs, for instance, \cite[Chapter~XVII]{DaLi92V} or \cite{Ou04, ReSi78IV}. Note that, by a semigroup, we always mean a $C_0$-semigroup.

The semigroup $T$ is said to be {\em real} if each operator $T(t)$ leaves the real part of $L^2_\mu(\Omega)$ invariant and it is called {\em positive} if each $T(t)$ leaves the positive cone of $L^2_\mu (\Omega )$ invariant. Clearly, a positive semigroup is always real. By \cite[Proposition~2.5]{Ou04}, the semigroup $T$ is real if and only if 
\[
 \forall\ u\in D(\E ) : \Re\ u \in D(\E ) \text{ and } \E (\Re\ u ,\text{Im }u ) \in\R ,
\]
Moreover, by the well-known {\em first Beurling-Deny criterion} (see, for instance, \cite[Theorem~1.3.2]{Da89}, \cite[Theorem~XIII.50]{ReSi78IV}, or \cite[Theorem~2.7]{Ou04}), positivity of a real semigroup $T$ is characterised by the property 
\[
 \forall\ u \in D(\E) : \, |u| \in D(\E) \text{ and } \E(|u|) \leq \E(u) .
\]

Let ${\mathcal A}$ be a closed, unital, and Hermitian subalgebra  of $C^b (\Omega )$, the space of bounded continuous functions. Here, as usual, an algebra is called \emph{Hermitian} if the spectrum of each of its self-adjoint elements is real.
We say that a closed, accretive, and sesquilinear form $\E$ is {\em $\mathcal A$-regular} if 
\begin{align*}
  & D(\E ) \cap {\mathcal A} \text{ is dense in the Hilbert space } D(\E ) \text{ and} \\
  & {\rm span}\,\big((D(\E ) \cap {\mathcal A}) \cup \{ 1\}\big) \text{ is dense in } \mathcal{A} \text{ with the supremum norm.}
\end{align*}

In this note, we use an abstract compactification $\widetilde{\Omega}$ of $\Omega$, an abstract boundary of $\Omega$, and a capacity on $\widetilde{\Omega}\times\widetilde{\Omega}$, which were recently considered in the case of nonlinear Dirichlet forms by Claus \cite{Cl21}. Let us explain this in more detail. First, by the Gelfand representation theorem, the Banach algebra $\mathcal{A}$ is isometrically isomorphic to the Banach algebra $C(\widetilde{\Omega})$ for some compact space $\widetilde{\Omega}$. This compact space $\widetilde{\Omega}$ is a compactification of $\Omega$. The construction of the above isometric isomorphism -- via the so-called Gelfand space or Gelfand spectrum -- shows that there is a natural, continuous mapping 
\begin{align*}
 \iota : \Omega & \to \widetilde{\Omega} , \\
  x & \mapsto \iota (x) ,
\end{align*}
such that the inverse of the above-mentioned isometric isomorphism is given by 
\begin{align*}
 J : C(\widetilde{\Omega}) & \to \mathcal{A} , \\
 f & \mapsto f\circ\iota .
\end{align*}
The mapping $\iota$ allows us to define a push-forward Borel measure $\widetilde{\mu}$ on $\widetilde{\Omega}$ by setting
\[
 \widetilde{\mu} (B) := \mu (\iota^{-1} (B))
\]
for every  Borel set $B\subseteq \widetilde{\Omega}$.
Now, the mapping $J$ above extends to an isometric isomorphism
\begin{align*}
 J : L^2_{\widetilde{\mu}} (\widetilde{\Omega}) & \to L^2_\mu (\Omega ) , \\
  f & \mapsto f\circ\iota .
\end{align*}
This means that in the general setting above, if the form $\E$ is $\mathcal A$-regular, then we may assume without loss of generality that $\Omega$ (actually, $\widetilde{\Omega}$) is a compact, topological space, $D(\E )$ (actually, $J^{-1}D(\E )$) is a dense subspace of $L^2 (\Omega )$, and $D(\E )\cap C(\Omega)$ is dense in $D(\E )$ and a fortiori in $L^2_\mu (\Omega )$.

In order to prove Theorem~\ref{main}, we freely use the theory of tensor norms. For this, we refer to the monograph \cite{DeFl93}. For example, we use the projective tensor product $D(\E )\otimes_\pi D(\E )$, but later also the injective tensor product $C(\widetilde{\Omega} ) \otimes_\varepsilon C(\widetilde{\Omega})$. Note that the projective tensor product $D(\E )\otimes_\pi D(\E )$ is a subspace of the Hilbert space tensor product $D(\E )\otimes_H D(\E )$. The latter is a subspace of the Hilbert space tensor product $L^2_\mu (\widetilde{\Omega}) \otimes_H L^2_\mu (\widetilde{\Omega})$ which in turn is isometrically isomorphic to $L^2_{\mu\otimes\mu} (\widetilde{\Omega} \times \widetilde{\Omega} )$, where $\mu\otimes\mu$ is the product measure. Hence, elements of $D(\E )\otimes_\pi D(\E )$ can be identified with (equivalence classes of) functions on the product space $\widetilde{\Omega} \times \widetilde{\Omega}$.

If $\E$ is a closed and accretive form, then we define the {\em $\E$-capacity} on the product space $\widetilde{\Omega}\times\widetilde{\Omega}$ by setting, for every subset $B\subseteq\widetilde{\Omega}\times\widetilde{\Omega}$,
\begin{align*}
    \mathcal{L}_B&:=\{ w \in L^2_{\mu\otimes\mu} (\widetilde{\Omega}\times\widetilde{\Omega}): w\geq 1 \,\, \mu\otimes\mu\text{-a.e. on an open set } U\supseteq B\} ,
    \\
    \Cap (B) &:= \inf \{ \| w\|_{D(\E ) \otimes_\pi D(\E )} :\ w \in \mathcal{L}_B \cap (D(\E ) \otimes_\pi D(\E )) \}.
\end{align*}
Here, $\| \cdot\|_{D(\E ) \otimes_\pi D(\E )}$ is the usual projective tensor norm, which on the algebraic tensor product is given by
\[
 \| w\|_{D(\E ) \otimes_\pi D(\E )} := \inf \left\{ \sum_{i=1}^n \| u_i\|_{D(\E )} \, \| v_i \|_{D(\E )} \big|\ n\in\N \text{ and } \sum_{i=1}^n u_i \otimes v_i = w \right\}.
\]
The projective tensor product is the completion of the algebraic tensor prod\-uct with respect to this norm. Our definition of the capacity is perhaps unusual in two aspects: first, our assumption on the form is minimal (accretive, closed) in order to define some capacity, and second, we do not define the capacity on $\widetilde{\Omega}$ but on the product space $\widetilde{\Omega}\times\widetilde{\Omega}$. However, when dealing with non-local forms, it is natural to define the capacity on the product space.

    From the above definition of capacity and properties of infimum, we immediately have that $\Cap(A)\le \Cap(B)$ for all $A$, $B\subseteq\widetilde{\Omega}\times\widetilde{\Omega}$ satisfying $A\subseteq B$.  More importantly, we have the following properties:
    
    \begin{lemma}
        \label{lem:subaddivity}
        Let $A$, $B\subseteq\widetilde{\Omega}\times\widetilde{\Omega}$. 
        \begin{enumerate}[(a)]
            \item There exists $C>0$ such that $(\mu \otimes \mu)(B)\le C \Cap(B)$.
            \item \emph{Subadditivity:} $\Cap(A\cup B)\le \Cap(A)+\Cap(B).$
        \end{enumerate}
    \end{lemma}

    \begin{proof}
        First of all, due to the continuous embedding of $D(\E ) \otimes_\pi D(\E )$ in the space $L^2_{\mu\otimes\mu}(\widetilde{\Omega}\times\widetilde{\Omega})$, 
        there exists $C>0$ such that
        \[
            (\mu \otimes \mu)(B) \le \|w\|_{L^2_{\mu\otimes\mu}(\widetilde{\Omega}\times\widetilde{\Omega})} \le C \| w\|_{D(\E ) \otimes_\pi D(\E )}.
        \]
        for all $w\in \mathcal{L}_B\cap (D(\E ) \otimes_\pi D(\E ))$. 
        Taking infimum over all such $w$ yields (a).
        
        (b) Let $w_A \in \mathcal{L}_A\cap (D(\E ) \otimes_\pi D(\E ))$ and $w_B\in\mathcal{L}_B\cap (D(\E ) \otimes_\pi D(\E ))$. Then $w_A + w_B \in \mathcal{L}_{A\cup B}\cap (D(\E ) \otimes_\pi D(\E ))$ and
        \[
            \Cap(A\cup B) \le \|w_A + w_B\|_{D(\E ) \otimes_\pi D(\E )} \le \| w_A\|_{D(\E ) \otimes_\pi D(\E )} + \| w_B\|_{D(\E ) \otimes_\pi D(\E )},
        \]
        from which the assertion readily follows.
    \end{proof}

A Borel measure $\nu$ on the product space $\widetilde{\Omega}\times\widetilde{\Omega}$ is called {\em symmetric} if $\nu (B) = \nu (\tilde{B} )$ for every Borel set $B$, where $\tilde{B} := \{ (x,y) \in\widetilde{\Omega}\times\widetilde{\Omega} | (y,x)\in B\}$ is the reflection of $B$. We say that a subset $B\subseteq\widetilde{\Omega}\times\widetilde{\Omega}$ is  {\em $\E$-polar} if $\Cap (B) = 0$ and the measure $\nu$ is said to be {\em absolutely continuous with respect to the $\E$-capacity} if $\nu (B) = 0$ for every $\E$-polar Borel set $B\subseteq\widetilde{\Omega}\times\widetilde{\Omega}$. Let us mention that if the capacity and the measure $\nu$ were only defined on $\widetilde{\Omega}$, then one could also define absolute continuity of the measure with respect to the capacity; this property was called {\em admissibility} in \cite{ArWa03b}.

    We also need the concept of quasi-continuity. Let $Y$ be a topological space. A function $f:\widetilde{\Omega}\times\widetilde{\Omega} \to Y$ is said to be \emph{quasi-continuous} if for each $\epsilon>0$, there exists an open set $U\subseteq \widetilde{\Omega}\times\widetilde{\Omega}$ such that $\Cap(U)<\epsilon$ and $f\restrict{U^c}$ is continuous. In particular, $f\in L^2_{\mu\otimes\mu}(\widetilde{\Omega}\times\widetilde{\Omega})$ is said to be \emph{quasi-continuous} if it has a representative (which we again denote by $f$) which is quasi-continuous. Finally, a property is said  to hold \emph{quasi-everywhere (q.e.)} if it holds everywhere except possibly on an $\E$-polar set.

\section{Proof of the main result}
\label{section:proof}

In this section, we provide a proof of Theorem~\ref{main}. 
First, we make a brief remark on the boundedness condition \eqref{eq.domination.2}.

\begin{remark} \label{rem.dom}
 In some typical examples -- for instance, forms associated with the Laplace operator with local boundary conditions -- the sesquilinear form $\widehat{\E}$ satisfies $\widehat{\E} (u,v) = \widehat{\E} (|u|,|v|)$ for every real $u$, $v\in D(\E)$ such that $uv\geq 0$. In such examples, the boundedness condition \eqref{eq.domination.2} implies the condition
 \begin{equation} \label{eq.domination.3}
  \int_{\widetilde{\Omega}\times\widetilde{\Omega}} u(x) v(y) \; d\nu (x,y) \geq 0 \text{ for every real } u,v\in D(\E ) \text{ such that } uv\geq 0 .
 \end{equation}
 Note that the product $uv$ at the end of \eqref{eq.domination.3} is a function on $\widetilde{\Omega}$, while the tensor product $u\otimes v$ appearing under the integral is a function on $\widetilde{\Omega}\times\widetilde{\Omega}$. The positivity of the product $uv$ only means that the tensor product $u\otimes v$ is positive on the diagonal $\Delta := \{ (x,y)\in\widetilde{\Omega}\times\widetilde{\Omega} | x=y\}$. The condition \eqref{eq.domination.3} can be rewritten in the form
 \begin{equation} \label{eq.domination.4}
  \begin{split}
   \int_\Delta u(x) v(x) \; d\nu (x,x) & \geq - \int_{\widetilde{\Omega}\times\widetilde{\Omega} \setminus\Delta} u(x) v(y) \; d\nu (x,y) \\
   & \text{ for every real } u,v\in D(\E ) \text{ such that } uv\geq 0 .
  \end{split}
 \end{equation}
 We say that the measure $\nu$ is {\em diagonally dominant} if it satisfies \eqref{eq.domination.4}. Diagonal dominance and also the condition \eqref{eq.domination.2} is a certain boundedness condition on $\nu$.
 
 Note that, there are positive measures $\nu$ which are not diagonally dominant. The sesquilinear form $\mathfrak{b} : H^1 (0,1) \times H^1 (0,1)\to\C$ given by 
 \[
  \mathfrak{b} (u,v) = \int_0^\frac12 u \cdot \int_\frac12^1 \bar{v} + \int_0^\frac12 \bar{v} \cdot \int_\frac12^1 u
 \]
is well defined, continuous, and positive on $H^1 (0,1)$, it is of the form of the integral term in \eqref{eq.domination.3} for the Lebesgue measure $\nu$ restricted to the union of the squares $\left(\left[0,\frac12\right]\times \left[\frac12 ,1\right]\right) \cup \left([\frac12 ,1]\times [0,\frac12 ]\right)$, but this measure $\nu$ is not diagonally dominant in the sense of condition \eqref{eq.domination.4}; indeed this can be seen by taking $u = v$ where $u(x) = 1$ for $x\in \left[0,\frac12\right]$ and $u(x) = -1$ for $x\in \left(\frac12 , 1\right]$. 
\end{remark}

The importance of the following technical result will become clear in Remark~\ref{rem:well-defined-integrals}. The proof mimics the arguments of \cite[Theorem~2.1.3]{FuOsTa11}.

    \begin{proposition}
        \label{prop:quasi-continuity}
        The functions in $D(\E ) \otimes_\pi D(\E )$ are quasi-continuous, that is, their equivalence classes contain quasi-continuous representatives.
    \end{proposition}

    \begin{proof}
        First of all, we assert that the inequality
        \begin{equation}
        \label{eq:quasi-continuous-sufficient}
            \Cap\{ |w| >\lambda \} \le \frac{ \| w\|_{D(\E ) \otimes_\pi D(\E )} }{2\lambda}
        \end{equation}
        holds for each $\lambda>0$ and all continuous $w\in D(\E ) \otimes_\pi D(\E )$. 
        Indeed, let $\lambda>0$ and $w \in D(\E ) \otimes_\pi D(\E )$ be continuous. Then $A=\{w>\lambda\}$ is open and $\frac{w}{\lambda} \in\mathcal{L}_A \cap (D(\E ) \otimes_\pi D(\E ))$. Therefore, $\Cap(A)\le \frac{\norm{w}_{D(\E ) \otimes_\pi D(\E )}}{\lambda}$. Analogously, taking $B=\{w<-\lambda\}$ satisfies $\Cap(B)\le \frac{\norm{w}_{D(\E ) \otimes_\pi D(\E )}}{\lambda}$. The inequality~\eqref{eq:quasi-continuous-sufficient} now follows using subaddivity (Lemma~\ref{lem:subaddivity}(b)).

        Next, let $w\in D(\E ) \otimes_\pi D(\E )$ be arbitrary. By the regularity of $\E$, there exists a sequence of continuous functions $(w_n)$ in $D(\E ) \otimes_\pi D(\E )$ such that $w_n\to w$ in $D(\E ) \otimes_\pi D(\E )$. Replacing by a subsequence, we may assume without loss of generality that $\norm{w_{k+1}-w_k}_{D(\E ) \otimes_\pi D(\E )} \le 2^{-2k}$ holds for each $k\in \mathbb N$. Thus, $G_k:= \{ |w_{k+1}-w_k| >2^{-k}\}$ satisfies $\Cap(G_k)\le 2^{-(k+1)}$ by~\eqref{eq:quasi-continuous-sufficient}. Set $F_k:= \bigcap_{j=k}^{\infty} G_j^c$. Then $(F_k)$ is a sequence of increasing closed sets such that $\lim_{k\to \infty} \Cap(F_k^c)=0$. Moreover, on $F_k$, we have the inequality
        \[
            |w_n-w_m| \le \sum_{j=k}^{\infty} |w_{j+1}-w_j|\le 2^{-(k-1)}
        \]
        for each $n$, $m\ge k$. Therefore, for each $k$, the sequence $(w_n\restrict{F_k})$ is uniformly convergent on $F_k$. Setting $\tilde w(x)=\lim w_n(x)$ for $x\in \bigcup_{k=1}^\infty F_k$, we get that $\tilde w$ is continuous on each $F_k$ and $w=\tilde w$ quasi-everywhere and hence also a.e. (Lemma~\ref{lem:subaddivity}(a)).  As $\lim_{k\to \infty} \Cap(F_k^c)=0$, it follows that $\tilde{w}$ is quasi-continuous. 
    \end{proof}

Next, we outsource the proof of  the necessity of the measure $\nu$ to be absolutely continuous with respect to the $\E$-capacity to the following lemma.

\begin{lemma} \label{lem.abs}
 Let the sesquilinear form $\E$, the algebra ${\mathcal A}$, and the compactification $\widetilde{\Omega}$  be as in Theorem~\ref{main}. Further, let $\nu$ be a positive Borel measure on $\widetilde{\Omega}\times \widetilde{\Omega}$ and let the form $\mathfrak{b} : D(\E ) \times D(\E ) \to\C$ be given by
\[
 \mathfrak{b} (u,v) = \int_{\widetilde{\Omega}\times \widetilde{\Omega}} u(x) \overline{v(y)} \; d\nu (x,y)
\]
for all continuous $u,v\in D(\E )$.

Assume that $\mathfrak b$ is continuous on $D(\E )$, i.e., there exists $C>0$ such that the inequality $|\mathfrak{b} (u,v)|\leq C\, \| u\|_{D(\E )} \| v\|_{D(\E )}$ is true for every $u$, $v\in D(\E )$. Then $\nu (B) \leq C\, \Cap (B)$ for every $B\subseteq \widetilde{\Omega}\times\widetilde{\Omega}$. In particular, $\nu$ is absolutely continuous with respect to the $\E$-capacity.
\end{lemma}

\begin{proof}
 Let $B\subseteq \widetilde{\Omega}\times\widetilde{\Omega}$. Then, for every non-negative $w\in D(\E ) \otimes_{alg} D(\E )$ with $w = \sum_{i=1}^n u_i \otimes v_i$ for continuous $u_i, v_i \in D(\E)$ and $w \geq 1$ on a neighbourhood of $B$, we have
 \begin{align*}
  \nu (B) & \leq \int_{\widetilde{\Omega}\times\widetilde{\Omega}} w(x,y) \; d\nu (x,y) \\
  & = \sum_{i=1}^n \int_{\widetilde{\Omega}\times\widetilde{\Omega}} u_i(x) v_i(y) \; d\nu (x,y)
 \end{align*}
 Since $b$ is continuous on $D(\E)$, the above inequality yields 
 \[
    \nu(B)\le \sum_{i=1}^n C\, \| u_i\|_{D(\E )} \, \| v_i\|_{D(\E )}.
\]
The regularity of the form  and \cite[Proposition~3.9 on page~36]{DeFl93} implies now that $\nu(B)\le C\|w\|_{D(\E ) \otimes_\pi D(\E )}$ for all $w\in D(\E ) \otimes_{alg} D(\E )$ with $w\ge 1$ on a neighbourhood of $B$. 
Taking the infimum over all such $w$, the assertion
$
 \nu (B) \leq C\, \Cap (B) ,
$
follows.
\end{proof}

\begin{remark}
    \label{rem:well-defined-integrals}
    We point out that when speaking of an integral of the elementary tensors $u \otimes v\in D(\E )\otimes D(\E)$ with respect to a measure $\nu$ that is absolutely continuous with respect to the $\E$-capacity, we implicitly take the quasi-continuous representatives of $u$ and $v$ (this is possible due to Proposition~\ref{prop:quasi-continuity}).
\end{remark}

\begin{proof}[Proof of Theorem~\ref{main}]
Recall that, by assumption, both the semigroups $T$ and $\widehat{T}$ are real. Also, by \cite[Theorem~2.21]{Ou04}, the semigroup $T$ is dominated by $\widehat{T}$  if and only if $D(\E )$ is an ideal of $D(\widehat{\E})$ and
\begin{equation} \label{eq.domination.char}
     \widehat{\E} (|u|,|v|) \leq \Re\ \E (u,v)
\end{equation}
for every $u$, $v\in D(\E )$ with $u\bar{v}\geq 0$.

Now, if $D(\E )$ is an ideal of $D(\widehat{\E} )$ and if $\E$ is given as in equation \eqref{eq.domination} for some appropriate measure $\nu$, which is absolutely continuous with respect to the $\E$-capacity, and which satisfies the lower bound \eqref{eq.domination.2}, then
\begin{align*}
    \widehat{\E} (|u|,|v|) & \leq\Re\  \widehat{\E} (u,v) +\Re\  \int_{\widetilde{\Omega}\times \widetilde{\Omega}} u(x) \overline{v(y)} \; d\nu (x,y) \\
 & =\Re\  \E (u,v) . 
\end{align*}
for all $u$, $v\in D(\E )$ with $u\bar{v}\geq 0$,
Hence, the inequality \eqref{eq.domination.char} is fulfilled for every $u$, $v\in D(\E )$ such that $u\bar{v}\geq 0$. Therefore, $T$ is dominated by $\widehat{T}$. 

Conversely, suppose that $T$ is dominated by $\widehat{T}$. In particular, the inequality \eqref{eq.domination.char} is fulfilled for every $u$, $v\in D(\E )$ such that $u\bar{v}\geq 0$. Define the sesquilinear form
\begin{align*} 
 \Psi : D(\E ) \times D(\E ) & \to \C , \\
 (u,v) & \mapsto \E (u,v) - \widehat{\E} (u,v) .
\end{align*}
Let $u,v$ be positive elements in $D(\E)$. In particular, $u,v$ are real. Since the semigroups $T$ and $\widehat{T}$ are real, therefore by a characterisation for real semigroups \cite[Proposition~2.5]{Ou04}, we obtain that $\E(u,v)$ and $\widehat{\E} (u,v)$ are also real.
Thus, the inequality \eqref{eq.domination.char} implies that
\[
 \Psi (u,v) \geq 0 ,
\]
that is, $\Psi$ is positive on $D(\E ) \times D(\E )$ in the order sense. By the universal property of the algebraic tensor product $D(\E ) \otimes_{\alg} D(\E )$, associated with the sesquilinear form $\Psi$, a unique linear mapping $\psi : D(\E ) \otimes_\alg D(\E )\to\C$ exists such that the diagram
\[
   \begin{tikzcd}
    D(\E ) \times D(\E ) \arrow{r}{\mathfrak{b}} \arrow[swap]{dr}{\Psi} & D(\E ) \otimes_\alg D(\E ) \arrow{d}{\psi} \\
     & \C
  \end{tikzcd}
\]
commutes. Here, $\mathfrak{b} : D(\E ) \times D(\E ) \to D(\E ) \otimes_\alg D(\E )$, $(u,v) \mapsto u\otimes \bar{v}$ is the canonical sesquilinear form. By the construction of the algebraic tensor product and by the positivity of $\Psi$, we obtain that $\psi(u otimes v)$ is positive whenever $u$ and $v$ are positive. This readily implies that $\psi$ is a positive functional on $D(\E ) \otimes_\alg D(\E )$ (see, for instance, \cite[Section~7]{Fr72}).
In particular,  $\psi$ also positive on the smaller tensor product $(D(\E ) \cap {\mathcal A}) \otimes_\alg (D(\E ) \cap {\mathcal A})$.

In what follows, we write $C(\widetilde{\Omega})$ instead of $\mathcal A$; where $\widetilde{\Omega}$ denotes the compactification of $\Omega$ with respect to $\mathcal{A}$.
We consider two cases. Assume first, that the span of $D(\E ) \cap C(\widetilde{\Omega})$ is dense in $C(\widetilde{\Omega})$ with respect to the supremum norm. Then there exists a positive $u_0\in D(\E )\cap C(\widetilde{\Omega})$ such that $\| u_0 - 1\|_\infty \leq \frac12$. This function $u_0$ is an order unit in $C(\widetilde{\Omega})$ and similarly, $u_0\otimes u_0$ is an order unit in the injective tensor product $C(\widetilde{\Omega})\otimes_\varepsilon C(\widetilde{\Omega} )$; the injective tensor product actually is isometrically isomorphic to the space $C(\widetilde{\Omega} \times \widetilde{\Omega} )$ (see \cite[Example I.4.2(3)]{DeFl93}). The tensor product $(D(\E ) \cap C(\widetilde{\Omega})) \otimes_\alg (D(\E ) \cap C(\widetilde{\Omega}))$ therefore is majorizing in $C(\widetilde{\Omega} \times \widetilde{\Omega} )$ in the sense that for every $w\in C(\widetilde{\Omega} \times \widetilde{\Omega} )$ there exists $z\in (D(\E ) \cap C(\widetilde{\Omega})) \otimes_\alg (D(\E ) \cap C(\widetilde{\Omega}))$ such that $w\leq z$. The positivity of the mapping $\psi$ and Kantorovich's theorem \cite[Corollary 1.5.9]{MN91} now imply that $\psi$ uniquely extends to a positive and bounded linear form on the space $C(\widetilde{\Omega} \times \widetilde{\Omega} )$; actually, Kantorovich's theorem in this special situation is an exercise. By the Riesz-Markov representation theorem, there exists a unique, positive, and finite measure $\nu$ on $\widetilde{\Omega}\times\widetilde{\Omega}$ such that 
\begin{equation} \label{representation}
 \psi (u\otimes \bar{v}) = \Psi (u,v) = \int_{\widetilde{\Omega}\times \widetilde{\Omega}} u(x) \overline{v(y)} \; d\nu (x,y) ,
\end{equation}
and the implication is proved.

Secondly, assume that the span $D(\E ) \cap C(\widetilde{\Omega})$ is not dense in $C(\widetilde{\Omega})$. By assumption of $\mathcal A$-regularity, however, the span of $\big(D(\E ) \cap C(\widetilde{\Omega}) \cup \{1\}\big)$ is dense in $C(\widetilde{\Omega} )$. Hence, the closure $\mathcal A_0$ of the space $D(\E ) \cap C(\widetilde{\Omega})$ has co-dimension one in the space of continuous functions. The fact that $D(\E )$ is a lattice implies that $\mathcal A_0$ is actually a maximal ideal in $C(\widetilde{\Omega} )$. Hence, there exists a vector $x_\infty\in\widetilde{\Omega}$ such that  
\[
\mathcal A_0 = \{ u\in C(\widetilde{\Omega} ) | u(x_\infty ) = 0 \} = C_0 (\widetilde{\Omega}_\infty );
\]
where $\widetilde{\Omega}_\infty := \widetilde{\Omega} \setminus \{x_\infty\}$. Similarly as above, one shows that the tensor product $(D(\E ) \cap C(\widetilde{\Omega})) \otimes_\alg (D(\E ) \cap C(\widetilde{\Omega}))$ is majorizing in $C_c(\widetilde{\Omega}_\infty \times \widetilde{\Omega}_\infty )$, and therefore, $\psi$ uniquely extends to a positive and bounded linear form on $C_c(\widetilde{\Omega}_\infty \times \widetilde{\Omega}_\infty )$. By the Riesz-Markov representation theorem, there exists a unique, positive Radon measure $\nu$ on $\widetilde{\Omega}_\infty$ such that \eqref{representation} holds. Note that in this case, the measure $\nu$ need not be finite, but for all functions $u$, $v\in D(\E )$ the product $u\otimes \bar{v}$ is $\nu$-integrable.

Now, the measure $\nu$ is positive and the forms $\E$, $\widehat{\E}$, and in turn, $\Psi$ are Hermitian. This implies that the measure $\nu$ can be chosen to be symmetric. In fact, if necessary, it suffices to replace the measure $\nu$ by the measure $\tilde{\nu}$ given by $\tilde{\nu} (B) = \frac12 (\nu (B) + \nu (\tilde{B}))$ for every Borel set $B\subseteq\widetilde{\Omega}\times\widetilde{\Omega}$ (where again, $\tilde{B}$ is the reflection of $B$). Finally, the continuity of $\psi$ on the space $D(\E ) \otimes_\alg D(\E )$ together with Lemma~\ref{lem.abs} implies that $\nu$ is absolutely continuous with respect to the $\E$-capacity. 
\end{proof}

\section{Positivity and locality}
\label{section:locality}

In a recent article, Akhlil found a surprising connection between the positivity of a dominated semigroup and the locality of the generating form, at least when the dominating semigroup is generated by a local operator or a local form \cite[Theorem~4.3]{Ak18}. Let us recall, that an operator $A$ on $L^2_\mu (\Omega )$ is called {\em local}, if for every $u\in D(A)$, $v\in L^2_\mu (\Omega )$ satisfying $uv = 0$ one has $\langle Au , v\rangle_{L^2_\mu} = 0$. Note that the condition $uv=0$ for two $L^2$-functions is equivalent to the condition $|u|\wedge |v| = 0$. Similarly, we say that a sesquilinear form $\E$ is {\em local}, if for every $u$, $v\in D(\E )$ satisfying $uv = 0$ one has $\E (u,v) = 0$. This is equivalent to the condition that for every $u$, $v\in D(\E )$ satisfying $uv = 0$ one has
\begin{equation} \label{eq.local}
 \E (u+v) = \E (u) + \E (v). 
\end{equation}
The proof of the following is straightforward.

\begin{proposition}
 Let $\E$ be a closed, accretive, and $\mathcal A$-regular sesquilinear form on $L^2_\mu (\Omega )$; where $\mathcal A$ is a closed, Hermitian, and unital subalgebra  of $C^b (\Omega )$. Let $\widetilde{\Omega}$ be the compactification of $\Omega$ associated with $\mathcal{A}$ and let $\nu$ be a positive Borel measure on $\widetilde{\Omega} \times \widetilde{\Omega}$ such that the form
 \[
  {\mathfrak b} (u,v) = \int_{\widetilde{\Omega}\times\widetilde{\Omega}} u(x)\overline{v(y)} \; d\nu (x,y) 
 \]
is well defined and continuous on $D(\E )$. Then $\mathfrak b$ is local if and only if ${\rm supp}\, \nu$ is contained in the diagonal $\Delta := \{ (x,y)\in\widetilde{\Omega}\times\widetilde{\Omega} | x=y\}$. 
\end{proposition}

The first statement in the following theorem gives a condition under which the locality of a form implies the positivity of the generated semigroup. The second statement in the theorem is basically \cite[Theorem~4.3]{Ak18}, although with weaker assumptions. The proof in \cite{Ak18} uses the Beurling-Deny and Lejan representation of regular Dirichlet forms \cite{Al75,And75}; in particular, the dominated semigroup is a submarkovian semigroup, that is, it is positive {\em and} $L^\infty$-contractive. We give here a different proof that only uses the positivity of the dominated semigroup.

\begin{theorem} \label{thm.local}
Let $\E :D(\E )\times D(\E )\to\C$ and $\widehat{\E} : D(\widehat{\E} ) \times D(\widehat{\E} ) \to\C$ be two sesquilinear Hermitian forms on $L^2_\mu (\Omega )$. Denote the associate semigroups by $T$ and $\widehat{T}$ respectively and assume they are real.
\begin{enumerate}[\upshape (a)]
 \item If the form $\E$ is local and $D(\E )$ is a sublattice of $L^2_\mu (\Omega )$, then $T$ is a positive semigroup.
 \item Assume that the semigroup $T$ is positive and $T$ is dominated by $\widehat{T}$. Then locality of $\widehat{\E}$ implies the locality of $\E$.
\end{enumerate}
\end{theorem}

In Section~\ref{sec.eventual-positivity}, we give an example to show that the positivity assumption in Theorem~\ref{thm.local}(b) cannot be dropped. In fact, we show that it is not even sufficient that the semigroup operators become (and remain) positive after a large enough time.

\begin{proof}
(a) Assume that the form $\E$ is local and that $D(\E )$ is a sublattice of $L^2_\mu (\Omega )$. Then, by the characterisation \eqref{eq.local} of locality,
 \begin{align*}
  \E (u) & = \E (u^+ - u^-) \\
  & = \E (u^+) + \E (-u^-) \\
  & = \E (u^+) + \E (u^-) \\
  & = \E (u^+ + u^-) \\
  & = \E (|u|) 
 \end{align*}
 for all $u\in D(\E)$. Thus, the first Beurling-Deny criterion implies that the semigroup $T$ is positive.

(b) Suppose that the form $\widehat{\E}$ generating $\widehat{T}$ is local. The positivity of $T$ and the first Beurling-Deny criterion imply, that for every real $u\in D(\E )$,
\begin{align*}
 \E (u^+) + \E (u^-) - 2\Re\ \E (u^+,u^-) & = \E (u) \\
 & \geq \E (|u|) \\
 & =  \E (u^+) + \E (u^-) + 2\Re\ \E (u^+,u^-) .
\end{align*}
Hence,
$
    \Re\  \E (u^+ ,u^-) \leq 0 .
$
for all real $u\in D(\E )$.
Thus, if $u$, $v\in D(\E )$ are both positive and $uv=0$, then the characterisation of domination from \eqref{eq.domination.char} implies
\[
 0 \geq \Re\ \E (u,v) \geq \Re\ \widehat{\E} (u,v) .
\]
However, because $\widehat{\E}$ is local, the right-hand side of this chain of inequalities is zero. Thus,
\[
\Re\ \E (u,v) = 0 \text{ for every positive } u,v\in D(\E ) \text{ such that } uv=0 .
\]
From the sesquilinearity of the form, we then obtain
\[
\Re\ \E (u,v) = 0 \text{ for every } u,v\in D(\E ) \text{ such that } uv= 0.
\]
Finally, replacing $u$ by $e^{i\theta} u$ yields 
\[
 \E (u,v) = 0 \text{ for every } u,v\in D(\E ) \text{ such that } uv= 0.
\]
Whence, the form $\E$ is local.
\end{proof}

\section{Final Remarks}
\label{section:remarks}

\subsection{The Laplace operator with Dirichlet and Neumann boundary conditions: what is in between?}

Let $T$, $\widehat{T}$, and $S$ be three semigroups generated by closed, accretive, and Hermitian
forms $\E$, $\widehat{\E}$, and $\mathfrak{b}$ respectively. If $S$ is \emph{sandwiched} between $T$ and $\widehat{T}$,  in the sense that $\widehat{T}$ dominates $S$ which in turn dominates $T$, then $S$ is necessarily positive. Indeed this holds because
\[
    |T(t)u| \leq S(t) |u|
\]
for every $u\in L^2_\mu(\Omega)$. Therefore Theorem~\ref{thm.local} implies that if $\widehat{\E}$ is local, then $\mathfrak{b}$ is local as well.

Let $\Omega \subset\mathbb{R}^{N}$ ($N \geq 1$) be a bounded open set with boundary $\partial \Omega$. By $\E^N : H^1 (\Omega) \times H^1 (\Omega) \to \R$ and $\E^D : H^1_0 (\Omega) \times H^1_0 (\Omega) \to \R$ we denote the sesquilinear and Hermitian forms of the Neumann-Laplace operator and the Dirichlet-Laplace operator, respectively, i.e.,
\begin{align*}
 \E^N (u,v) & := \int_{\Omega} \nabla u \overline{\nabla v} \; \mathrm{d}x \qquad (u, v\in H^1 (\Omega))  \text{ and} \\
 \E^D (u,v) & := \int_{\Omega} \nabla u \overline{\nabla v} \; \mathrm{d}x \qquad (u, v\in H^1_0 (\Omega)) . 
\end{align*}
Both forms $\E^N$ and $\E^D$ are local and $T^N$ dominates $T^D$, where $T^N$ and $T^D$ are the associated semigroups. Let $S$ be sandwiched between $T^N$ and $T^D$, generated by a closed, accretive, and Hermitian form $\mathfrak{b}$. Then (as remarked above), $\mathfrak{b}$ is necessarily local by Theorem~\ref{thm.local}. If $\Omega$ has Lipschitz boundary, so that $\mathfrak{b}$ is also $C(\overline{\Omega})$-regular, then the form is associated to a Laplace operator with local Robin boundary conditions \cite[Theorem~4.1]{ArWa03b}. As shown in Theorem~\ref{thm.local} (b) above, and as already shown by Akhlil in  \cite{Ak18}, the assumption of locality in \cite[Theorem~4.1]{ArWa03b} is superfluous. To sum up, \cite[Theorem~4.1]{ArWa03b} can be simplified as follows:

\begin{proposition}
    Let $S$ be a semigroup generated by a closed, accretive, and Hermitian form $\mathfrak{b}$ on $L^2_\mu(\Omega)$. Also, let $T^D$ and $T^N$ denote the semigroups generated by the Dirichlet Laplacian and Neumann Laplacian respectively.
    
    If $\Omega$ has Lipschitz boundary, then $S$ is sandwiched between $T^D$ and $T^N$ if and only if $S$ is generated by a Laplace operator with local Robin boundary conditions.
\end{proposition}

A similar representation theorem was also proved in a nonlinear setting in which closed, accretive, and Hermitian forms are replaced by convex, lower semi-continuous energy functions. For example, \cite{ChWa12} characterises all nonlinear semigroups which are sandwiched between the semigroups generated by the $p$-Laplace operator with Neumann boundary conditions and the $p$-Laplace operator with Dirichlet boundary conditions, if the energy function generating the sandwiched semigroup is local. Later, in \cite{Cl21b}, such a characterisation was generalised to semigroups generated by nonlinear and local Dirichlet forms. It is not clear whether locality is a necessary assumption in the nonlinear situation.

\subsection{Eventual positivity} \label{sec.eventual-positivity}

Let $\E:D(\E)\times D(\E)\to\mathbb C$ be a sesquilinear Hermitian form on $L^2_{\mu}(\Omega)$, where $\Omega\subseteq\R^N$ is a bounded open set.
As a consequence of Theorem~\ref{thm.local}(b), we have that, if the semigroup $T$ associated to $\E$ is dominated by the semigroup generated by the Neumann Laplacian (see above), then the positivity of $T$ implies locality of $\E$. However, there are non-positive semigroups that are dominated by the semigroup generated by the Neumann Laplacian. Of course, due to Theorem~\ref{thm.local}(a), they are necessarily non-local. We give an example:~Let $\Omega=(0,1)$ and let $T$ be the semigroup associated with the non-local form
\begin{align*}
    \E(u,v) & = \int_0^1 u'\bar{v'}\, dx+ \langle Bu\restrict{\{0,1\}},v\restrict{\{0,1\}}\rangle \\
    & = \int_0^1 u'\bar{v'}\, dx + \lambda (u(0)\overline{v(0)} + u(1)\overline{v(1)} + u(0)\overline{v(1)} + u(1)\overline{v(0)}  )
\end{align*}
for $u,v\in D(\E):=H^1(0,1)$; where $B=\begin{bmatrix} \lambda &\lambda\\ \lambda &\lambda\end{bmatrix}$ and $\lambda$ is a positive non-zero real number. The aforementioned domination is a consequence of Theorem~\ref{main}. Indeed, by Remark~\ref{rem.dom}, it suffices to show that the measure is diagonally dominant on $H^1(0,1)$. Note that the measure here is just the Dirac measure on the four corners of the unit square $[0,1]\times [0,1]$ and hence diagonally dominant. In fact, the domination of the semigroup $T$ by the Neumann Laplacian is also mentioned in \cite[Section~3]{Ak18} for the case $\lambda=1$. 

While non-positivity of the semigroup $T$ is a consequence of the first Beurling-Deny criterion, it can alternatively be deduced by Theorem~\ref{thm.local}(b).
Nevertheless, the semigroup $T$ is {\em uniformly eventually positive}, i.e., there exists a time $t_0\geq 0$ such that $T(t)$ leaves the positive cone invariant for all $t\geq t_0$. Indeed, this was shown for $\lambda=1$ in \cite[Theorem~4.2]{DaGl18a} and the proof for other values of $\lambda$ remains same. 

The first example of an eventually positive semigroup in infinite dimensions was given by Daners in \cite{Da14} which led to a systematic study in \cite{DaGlKe16, DaGlKe16a}. The theory has been further developed in \cite{DaGl17, DaGl18a, DaGl18, ArGl21}. Recently, the first author along with Glück showed that the semigroup $T$ above is \emph{eventually sandwiched} between the semigroups generated by the Dirichlet Laplacian and Neumann Laplacian \cite[Theorem~4.5]{ArGl22}.

\subsection*{Acknowledgements} 

The first and the third named author were supported by 
Deu\-tscher Aka\-de\-mi\-scher Aus\-tausch\-dienst. Part of the work was done during the third author's pleasant stay at TU Dresden.


\providecommand{\bysame}{\leavevmode\hbox to3em{\hrulefill}\thinspace}

\bibliographystyle{plainurl}
\bibliography{ralph}

\def\cprime{$'$} \def\cprime{$'$} \def\cprime{$'$} \def\cprime{$'$}
  \def\cprime{$'$} \def\cprime{$'$} \def\cprime{$'$} \def\cprime{$'$}
  \def\cprime{$'$} \def\cprime{$'$} \def\cprime{$'$} \def\cprime{$'$}
  \def\cprime{$'$} \def\cprime{$'$} \def\cprime{$'$} \def\cprime{$'$}
  \def\cprime{$'$} \def\cprime{$'$} \def\cprime{$'$} \def\cprime{$'$}
  \def\cprime{$'$} \def\cprime{$'$} \def\cprime{$'$} \def\cprime{$'$}
  \def\cprime{$'$} \def\cprime{$'$} \def\cprime{$'$} \def\cprime{$'$}
  \def\cprime{$'$} \def\cprime{$'$} \def\cprime{$'$}
  \def\ocirc#1{\ifmmode\setbox0=\hbox{$#1$}\dimen0=\ht0 \advance\dimen0
  by1pt\rlap{\hbox to\wd0{\hss\raise\dimen0
  \hbox{\hskip.2em$\scriptscriptstyle\circ$}\hss}}#1\else {\accent"17 #1}\fi}
  \def\cprime{$'$} \def\cprime{$'$} \def\cprime{$'$}
\begin{thebibliography}{10}

\bibitem{Ak18}
Khalid Akhlil.
\newblock Locality and domination of semigroups.
\newblock {\em Results Math.}, 73(2):Paper No. 59, 11, 2018.
\newblock \href {https://doi.org/10.1007/s00025-018-0822-9}
  {\path{doi:10.1007/s00025-018-0822-9}}.

\bibitem{Al75}
Guy Allain.
\newblock Sur la repr\'{e}sentation des formes de {D}irichlet.
\newblock {\em Ann. Inst. Fourier (Grenoble)}, 25(3-4):xix, 1--10, 1975.
\newblock \href {https://doi.org/10.5802/aif.570} {\path{doi:10.5802/aif.570}}.

\bibitem{And75}
Lars-Erik Andersson.
\newblock On the representation of {D}irichlet forms.
\newblock {\em Ann. Inst. Fourier (Grenoble)}, 25(3-4, {\rm xiii}):11--25,
  1975.
\newblock \href {https://doi.org/10.5802/aif.571} {\path{doi:10.5802/aif.571}}.

\bibitem{ArWa03b}
Wolfgang Arendt and Mahamadi Warma.
\newblock Dirichlet and {N}eumann boundary conditions: {W}hat is in between?
\newblock {\em J. Evol. Equ.}, 3(1):119--135, 2003.
\newblock Dedicated to Philippe B{\'e}nilan.
\newblock \href {https://doi.org/10.1007/978-3-0348-7924-8_6}
  {\path{doi:10.1007/978-3-0348-7924-8_6}}.

\bibitem{ArGl21}
Sahiba {Arora} and Jochen {Gl\"uck}.
\newblock {Spectrum and convergence of eventually positive operator
  semigroups}.
\newblock {\em {Semigroup Forum}}, 2021.
\newblock \href {https://doi.org/10.1007/s00233-021-10204-y}
  {\path{doi:10.1007/s00233-021-10204-y}}.

\bibitem{ArGl22}
Sahiba {Arora} and Jochen {Gl\"uck}.
\newblock Criteria for eventual domination of operator semigroups and
  resolvents.
\newblock {\em {To appear in Volume from IWOTA Lancaster 2021}}, 2022.
\newblock \href {http://arxiv.org/abs/2204.00146v3}
  {\path{arXiv:2204.00146v3}}.

\bibitem{ChWa12}
Ralph Chill and Mahamadi Warma.
\newblock Dirichlet and {N}eumann boundary conditions for the {$p$}-{L}aplace
  operator: what is in between?
\newblock {\em Proc. Roy. Soc. Edinburgh Sect. A}, 142(5):975--1002, 2012.
\newblock \href {https://doi.org/10.1017/S030821051100028X}
  {\path{doi:10.1017/S030821051100028X}}.

\bibitem{Cl21}
Burkhard Claus.
\newblock Energy spaces, {D}irichlet forms and capacities in a nonlinear
  setting.
\newblock {\em Potential Anal.}, 2021.
\newblock \href {https://doi.org/10.1007/s11118-021-09935-y}
  {\path{doi:10.1007/s11118-021-09935-y}}.

\bibitem{Cl21b}
Burkhard Claus.
\newblock {\em {Nonlinear Dirichlet forms}}.
\newblock PhD thesis, Technical University Dresden, Dresden, 2021.
\newblock \url{https://nbn-resolving.org/urn:nbn:de:bsz:14-qucosa2-758918}.

\bibitem{Da14}
Daniel {Daners}.
\newblock {Non-positivity of the semigroup generated by the
  Dirichlet-to-Neumann operator}.
\newblock {\em {Positivity}}, 18(2):235--256, 2014.
\newblock \href {https://doi.org/10.1007/s11117-013-0243-7}
  {\path{doi:10.1007/s11117-013-0243-7}}.

\bibitem{DaGl17}
Daniel Daners and Jochen Gl\"uck.
\newblock The role of domination and smoothing conditions in the theory of
  eventually positive semigroups.
\newblock {\em Bull. Aust. Math. Soc.}, 96(2):286--298, 2017.
\newblock \href {https://doi.org/10.1017/S0004972717000260}
  {\path{doi:10.1017/S0004972717000260}}.

\bibitem{DaGl18a}
Daniel {Daners} and Jochen {Gl\"uck}.
\newblock {A criterion for the uniform eventual positivity of operator
  semigroups}.
\newblock {\em {Integral Equations Oper. Theory}}, 90(4):19, 2018.
\newblock Id/No 46.
\newblock \href {https://doi.org/10.1007/s00020-018-2478-y}
  {\path{doi:10.1007/s00020-018-2478-y}}.

\bibitem{DaGl18}
Daniel {Daners} and Jochen {Gl\"uck}.
\newblock {Towards a perturbation theory for eventually positive semigroups}.
\newblock {\em {J. Oper. Theory}}, 79(2):345--372, 2018.
\newblock \href {https://doi.org/10.7900/jot.2017mar29.2148}
  {\path{doi:10.7900/jot.2017mar29.2148}}.

\bibitem{DaGlKe16a}
Daniel Daners, Jochen Gl\"uck, and James~B. Kennedy.
\newblock Eventually and asymptotically positive semigroups on {B}anach
  lattices.
\newblock {\em J. Differential Equations}, 261(5):2607--2649, 2016.
\newblock \href {https://doi.org/10.1016/j.jde.2016.05.007}
  {\path{doi:10.1016/j.jde.2016.05.007}}.

\bibitem{DaGlKe16}
Daniel Daners, Jochen Gl\"uck, and James~B. Kennedy.
\newblock Eventually positive semigroups of linear operators.
\newblock {\em J. Math. Anal. Appl.}, 433(2):1561--1593, 2016.
\newblock \href {https://doi.org/10.1016/j.jmaa.2015.08.050}
  {\path{doi:10.1016/j.jmaa.2015.08.050}}.

\bibitem{DaLi92V}
Robert Dautray and Jacques-Louis Lions.
\newblock {\em Mathematical analysis and numerical methods for science and
  technology. {V}ol. 5}.
\newblock Springer-Verlag, Berlin, 1992.
\newblock Evolution problems. I, With the collaboration of Michel Artola,
  Michel Cessenat and H{\'e}l{\`e}ne Lanchon, Translated from the French by
  Alan Craig.

\bibitem{Da89}
Edward~Brian Davies.
\newblock {\em Heat kernels and spectral theory}, volume~92 of {\em Cambridge
  Tracts in Mathematics}.
\newblock Cambridge University Press, Cambridge, 1989.
\newblock \href {https://doi.org/10.1017/CBO9780511566158}
  {\path{doi:10.1017/CBO9780511566158}}.

\bibitem{DeFl93}
Andreas Defant and Klaus Floret.
\newblock {\em Tensor norms and operator ideals}, volume 176 of {\em
  North-Holland Mathematics Studies}.
\newblock North-Holland Publishing Co., Amsterdam, 1993.

\bibitem{Fr72}
David~H. Fremlin.
\newblock Tensor products of {Archimedean} vector lattices.
\newblock {\em Am. J. Math.}, 94:777--798, 1972.
\newblock \href {https://doi.org/10.2307/2373758} {\path{doi:10.2307/2373758}}.

\bibitem{FuOsTa11}
Masatoshi Fukushima, Yoichi Oshima, and Masayoshi Takeda.
\newblock {\em Dirichlet forms and symmetric {M}arkov processes}, volume~19 of
  {\em de Gruyter Studies in Mathematics}.
\newblock Walter de Gruyter \& Co., Berlin, extended edition, 2011.
\newblock \href {https://doi.org/10.1515/9783110218091}
  {\path{doi:10.1515/9783110218091}}.

\bibitem{MN91}
Peter Meyer-Nieberg.
\newblock {\em Banach {L}attices}.
\newblock Springer Verlag, Berlin, Heidelberg, New York, 1991.
\newblock \href {https://doi.org/10.1007/978-3-642-76724-1}
  {\path{doi:10.1007/978-3-642-76724-1}}.

\bibitem{Ou04}
El~Maati Ouhabaz.
\newblock {\em Analysis of Heat Equations on Domains}, volume~30 of {\em London
  Mathematical Society Monographs}.
\newblock Princeton University Press, Princeton, 2004.
\newblock \href {https://doi.org/10.1515/9781400826483}
  {\path{doi:10.1515/9781400826483}}.

\bibitem{ReSi78IV}
Michael Reed and Barry Simon.
\newblock {\em Methods of modern mathematical physics. {IV}. {A}nalysis of
  operators}.
\newblock Academic Press [Harcourt Brace Jovanovich Publishers], New York,
  1978.

\end{thebibliography}

\end{document}